\documentclass[a4paper, 11pt]{article}
\usepackage[utf8]{inputenc}
\usepackage[margin=1.2in]{geometry} 
\usepackage{amsmath,amsthm,amssymb,amsfonts,amscd}
\usepackage{mathtools}
\usepackage{enumitem}
\usepackage{indentfirst}
\usepackage{algorithm, algpseudocode, algorithmicx}
\usepackage{appendix}
\usepackage{hyperref}
\setlist[enumerate,1]{label={(\alph*)}}
\title{The Doubly Stochastic Single Eigenvalue Problem: A Computational Approach}
\author{Amit Harlev\footnote{Harvey Mudd University, Claremont, California, USA.}\; , Charles R. Johnson\footnote{College of William and Mary, Williamsburg, Virginia, USA.}\; , and Derek Lim\footnote{Cornell University, Ithaca, New York, USA. dl772@cornell.edu (Corresponding author)}}
\date{April 3, 2020}

\newcommand{\RR}{\mathbb{R}}
\newcommand{\CC}{\mathbb{C}}
\newcommand{\mc}[1]{\mathcal{#1}}
\newcommand{\mrm}[1]{\mathrm{#1}}

\theoremstyle{plain}
\newtheorem{Thm}{Theorem}[section]
\newtheorem{lemma}[Thm]{Lemma}
\newtheorem{Cor}[Thm]{Corollary}

\theoremstyle{definition}

\newtheorem{Conj}[Thm]{Conjecture}
\newtheorem{observation}{Observation}[section]

\begin{document}
\maketitle

\begin{abstract}
    The problem of determining $DS_n$, the complex numbers that occur as an eigenvalue of an $n$-by-$n$ doubly stochastic matrix, has been a target of study for some time. The Perfect-Mirsky region, $PM_n$, is contained in $DS_n$, and is known to be exactly $DS_n$ for $n \leq 4$, but strictly contained within $DS_n$ for $n = 5$. Here, we present a Boundary Conjecture that asserts that the boundary of $DS_n$ is achieved by eigenvalues of convex combinations of pairs of (or single) permutation matrices. We present a method to efficiently compute a portion of $DS_n$, and obtain computational results that support the Boundary Conjecture. We also give evidence that $DS_n$ is equal to $PM_n$ for certain $n > 5$.
    \footnote{We include codes for running experiments and recreating figures in this paper in the following repository: \url{https://github.com/cptq/ds-spectra} }
\end{abstract}

    
    
\section{Introduction}
A matrix $A \in M_n(\RR)$ is \textit{doubly stochastic} if $A \geq 0$ (entry-wise) and all row and column sums of $A$ are $1$. Denote the set of doubly stochastic matrices by $\Omega_n$. The \textit{doubly stochastic single eigenvalue problem} asks for the determination of $DS_n = \bigcup_{A \in \Omega_n} \sigma(A)$, i.e. which complex numbers occur as an eigenvalue of an $n$-by-$n$ doubly stochastic matrix? This problem is unsettled for $n>4$ and is one of the several members of the family associated with the very difficult non-negative inverse eigenvalue problem (NIEP) \cite{NIEP}. Because permutation matrices are doubly stochastic, $DS_n$ includes all roots of unity of order $k \leq n$. Denote by $\Pi_k$ the convex hull of the \textit{k}-th roots of unity. In \cite{PM} it was noted that $PM_n = \bigcup_{k \leq n} \Pi_k \subseteq DS_n$, and that for $n<4$, $PM_n = DS_n$. Recently, it was shown in \cite{DS4} that $DS_4 = PM_4$. However, in \cite{MR}, a very particular matrix in $\Omega_5$ was exhibited with a conjugate pair of eigenvalues just outside of $PM_5$. We comment on this further, later. Of course $DS_5 \subseteq DS_6$, and this conjugate pair lies in $PM_6$, so that there is no implication about the relationship between $PM_n$ and $DS_n$ for $n>5$.

In contrast to $DS_n$, the row stochastic single eigenvalue problem (equivalent to the non-negative single eigenvalue problem \cite{NIEP}), is understood (see \cite{Karp}, \cite{DD}, \cite{Ito}, \cite{Matricial}). We refer to the region of eigenvalues achievable by row stochastic matrices as the Karpelevich region and denote it by $K_n$. Of course, this larger subset of the unit disc also includes the \textit{k}-th roots of unity, $k \leq n$ (which are the only points on the unit circle); between consecutive roots of unity, an in-bending algebraic curve excludes a portion of the unit disc.
Moreover, these curves follow an eigenvalue of a matrix on the line segment joining two simple row stochastic matrices \cite{Matricial}. We conjecture that the boundary of $DS_n$ (a set star-shaped from any $s \in [0,1]$) is determined by the eigenvalue paths resulting from the convex combinations of pairs of permutation matrices. Any doubly stochastic matrix is a convex combination of permutation matrices \cite{Birkhoff}, and in fact, as is clear by convexity theory, any $n$-by-$n$ doubly stochastic matrix is a convex combination of at most $(n-1)^2+1$  permutation matrices \cite{DiagonalsDS}.

Here, we present the Boundary Conjecture, which states that every complex number on the boundary of $DS_n$ is achieved as an eigenvalue of some doubly stochastic matrix that is a convex combination of two or fewer permutation matrices. Assuming this conjecture true significantly reduces the problem size of computing $DS_n$ to a given precision, and allows us to design algorithms much more efficient than naive methods for doing so. There is much theoretical evidence to support our conjecture; our purpose here is to give further computational evidence for our conjecture and to show through computation that, if it is correct, then it is likely that $DS_n = PM_n$ for $n=6,7,8,9,10$ and $11$ . Moreover, we investigate the exceptional $n=5$ case in detail. We find that the known example outside $PM_5$ is an eigenvalue of a matrix on the line segment joining two permutations and that a conjugate pair of eigenpaths are, up to uniform permutation similarity of the pair of permutation matrices generating it, the only paths arising from two permutations that leave $PM_5$. 

\section{Preliminaries and Notation}

We begin by stating Birkhoff's Theorem---a fundamental theorem underlying much of the study of doubly stochastic matrices \cite{Birkhoff}. Denote the convex hull of a set $S$ by $Co(S)$.
\begin{Thm}[Birkhoff]
$A \in M_n$ is doubly stochastic if and only if it is a convex combination of $n$-by-$n$ permutation matrices:
\[\Omega_n = Co(\{P \in M_n : \text{$P$ a permutation matrix}\})\]
\end{Thm}
Now, we prove some fundamental properties of $DS_n$. Much of this was shown in \\ \cite{PM}, though we shall for the most part present alternative proofs in hopes of inspiring new perspectives in which to study this difficult problem. In particular, we use some basic concepts from representation theory and do not use the Perron-Frobenius theory of nonnegative matrices.

First, observe that $DS_n$ is contained in the unit disc, due to the properties that the spectral radius is bounded by the spectral norm, convexity of the spectral norm and the fact that all permutation matrices have spectral norm $1$ \cite{MatrixAnalysis}. Moreover, observe that the permutation matrices share a common eigenvalue $1$ with eigenvector $e$, the all-ones vector. Thus, the technique of deflation can be used to reduce the problem size. We choose a deflating matrix $S$ with inverse of the following form:

\begin{equation}\label{deflation}
S = \begin{bmatrix}
1 & 0\\
e & I_{n-1}
\end{bmatrix}, \qquad S^{-1} = \begin{bmatrix}
1 & 0\\
-e & I_{n-1}
\end{bmatrix}
\end{equation}

\noindent so that multiplication by $S$ and its inverse can be done in linear time. Then we note that for any permutation matrix $P$, applying similarity by $S$ gives
\[S^{-1} P S = \begin{bmatrix}
1 & *\\
0 & P'
\end{bmatrix}\]

\noindent where $P' \in M_{n-1}$ is a (not generally nonnegative) matrix, that has the same multiset of eigenvalues as $P$, with the eigenvalue $1$ having a singly decremented multiplicity. 

For any convex combination of permutations, $\sum_i \alpha_i P_i$ has the same eigenvalues as\\ $\sum_i \alpha_i S^{-1} P_i S$, which has as lower right block the matrix $\sum_i \alpha_i P_i'$. The set of all $P'$ is called the standard representation of $S_n$, and is an irreducible representation of dimension $n-1$. Note that any eigenvalue $\lambda$ of a doubly stochastic matrix $A$ that is not equal to $1$ is thus an eigenvalue of the corresponding lower right block matrix of $S^{-1} A S$. Moreover, the eigenvalue $1$ is obviously achieved in the convex hull of the standard representation as an eigenvalue of say $I_{n-1}$. Thus, we have the following lemma:

\begin{lemma}\label{StandardRep}
$DS_n$ equals the region achieved by eigenvalues of the matrices in the convex hull of the standard representation of $S_n$.
\end{lemma}

A region $D \subseteq \CC$ of the complex plane is \textit{star-shaped} from a point $s \in \CC$ if for all points $x \in D$, the line segment $\{\alpha x + (1-\alpha) s : \alpha \in [0,1]\}$ is contained in $D$. It was shown in \cite{PM} that $DS_n$ is star-shaped from $0$---we will prove a slightly stronger statement.
\begin{lemma}\label{StarShaped}
If the scalar matrix $sI$ is in the convex hull $Co(\mc{F})$ of a family of matrices $\mc{F} \subseteq M_n$, then  the region of eigenvalues $\{\lambda : \lambda \in \sigma(A), A \in Co(\mc{F})\}$ is star-shaped from $s$.
\end{lemma}
\begin{proof}
If $\lambda$ is an eigenvalue of $A \in Co(\mc{F})$, then since $\alpha A + (1-\alpha)s I$ is in $Co(\mc{F})$ for $\alpha \in [0,1]$, we have that the eigenvalue $\alpha \lambda + (1-\alpha)s$ of this matrix is in the region of eigenvalues.
\end{proof}

\begin{Cor}\label{StarDSn}
$DS_n$ is star-shaped from any point in  $[0,1]$ for $n \geq 2$.
\end{Cor}
\begin{proof}
The standard representation of $S_n$ contains the identity $I_{n-1}$. Moreover, the average of the elements of the standard representation, $\frac{1}{n!}\sum_i P_i'$ is in the convex hull of the standard representation and is in fact equal to $0$, since the sum of all elements of any nontrivial irreducible representation is equal to $0$. Thus, all scalar matrices $s I$ for $s \in [0,1]$ are in the convex hull and Lemma \ref{StarShaped} allows us to draw the stated conclusion.
\end{proof}

The following lemmas are useful for our computations and provide simple quantitative measures of how close $DS_n$ is to filling the entire unit disc---in the limit $DS_n$ fills the interior of the unit disc and a dense subset (consisting of roots of unity) of the unit circle.

\begin{lemma}\label{Polygons}
$PM_n = \cup_{k \leq n} \Pi_k \subseteq DS_n$
\end{lemma}
\begin{proof}
A permutation matrix $A$ that is a k-cycle has minimal polynomial $x^k-1$. Any convex combination $\alpha_0 I + \alpha_1 A + \ldots + \alpha_{k-1}A^{k-1}$ is a polynomial in $A$ with eigenvalues $\alpha_0 + \alpha_1 \lambda + \ldots + \alpha_{k-1} \lambda^{k-1}$, where $\lambda$ is an eigenvalue of $A$. Since $\xi$, a primitive kth root of unity, is an eigenvalue of $A$, $\Pi_k = Co(1, \xi, \ldots, \xi^{k-1})$ is contained in $DS_n$.
\end{proof}

\begin{lemma}\label{inradiuslemma}
The circle of radius $\cos(\frac{\pi}{n})$, centered at the origin, is contained in $DS_n$ for $n > 2$.
\end{lemma}
\begin{proof}
	Consider the right triangle with vertices at $(0,0)$, $(1,0)$, and $M$, the midpoint of a side of $\Pi_n$ that has $(1,0)$ as an endpoint. The length of the line segment between $(0,0)$ and $M$ is the radius of the circle inscribed within $\Pi_n$. Since $\Pi_n$ is a regular polygon, the central angle is $2\pi/n$, so the right triangle gives that the radius of this inscribed circle is $\cos(\pi/n)$.
\end{proof}

Now, we formally state the Boundary Conjecture.

\begin{Conj}[Boundary Conjecture]
Pairs of permutation matrices determine the boundary of $DS_n$. That is, every point of the boundary of $DS_n$ is an eigenvalue of some convex combination of at most two permutation matrices.
\end{Conj}

If this conjecture were true, then by Corollary \ref{StarDSn}, the star-shapedness of $DS_n$ gives

\begin{Conj}
Every complex number in $DS_n$ is achieved as an eigenvalue of a convex combination of three or fewer permutations. More specifically, convex combinations of the form 
\[\alpha_1 P + \alpha_2 Q + (1-\alpha_1 - \alpha_2)I \text{\qquad $P,Q$ permutations, $\alpha_1, \alpha_2 \geq 0$, $\alpha_1 + \alpha_2 \leq 1 $}\]
achieve every eigenvalue in $DS_n$.
\end{Conj}

There is much evidence to suggest that the boundary of $DS_n$ is determined by pairs of permutations. For any $n$ such that $DS_n = PM_n$, this is true. This is due to the fact that, as can be seen by the reasoning in the proof of Lemma \ref{Polygons}, any edge of $\Pi_k$ is attained by eigenvalues along a suitable pair of a k-cycle and one of its powers. Moreover, any doubly stochastic matrix that has all eigenvalues of magnitude $1$ is in fact a permutation matrix \cite{PM}, so in this sense the matrices with the most extremal eigenvalues are the permutation matrices. Perhaps taking convex combinations of more than two permutations matrices has an averaging effect on the eigenvalues. In fact, in the standard representation, taking the average over all group elements gives the zero matrix (since it is an irreducible representation), which has all $0$ eigenvalues that are as far into the interior of $DS_n$ as possible. Further, the Karpelevich region consisting of all eigenvalues achievable by row stochastic matrices has boundary determined by convex combinations of pairs of certain row stochastic matrices \cite{Matricial}. Lastly, our computations have not found any eigenvalue on the boundary of $DS_n$ that does not belong to a convex combination of pairs of permutations.

In language and notation, we use the standard correspondence between permutations $\sigma \in S_n$ with matrices $P \in M_n$, where the matrix associated with $\sigma$ has entries $P_{ij} = 1$ if $\sigma(i) = j$ and $0$ otherwise. Moreover, we use standard cycle notation for permutations, and say a permutation in $S_n$ has cycle type $n_1, \ldots, n_k$ if it has disjoint cycles of length $n_1, \ldots, n_k$ and $\sum_i n_i = n$. Two sets of matrices $\mc{F}, \mc{F}' \subseteq M_n$ are \textit{uniformly similar} if there exists nonsingular $S$ such that $\mc{F} = \{S A S^{-1} : A \in \mc{F}'\}$. If $S$ can be taken to be a permutation matrix we say that the sets are \textit{uniformly permutation similar}. We make the analogous definition for sets of permutations in $S_n$ to be \textit{uniformly conjugate}.

Even though Perfect and Mirsky did not explicitly make such a statement in their original paper \cite{PM}, we refer to the Perfect-Mirsky conjecture for a given $n$ as the conjecture that $DS_n = PM_n$.

\section{The Computational Approach and Inequivalent Pairs}

In later sections we discuss some of the computational evidence that supports the Boundary Conjecture. In this section, we outline the computational approach to obtaining such evidence as well as information about $DS_n$ for $n \geq 5$. We search for counterexamples to the Perfect-Mirsky conjecture by computing the eigenvalues of those doubly stochastic matrices that are convex combinations of pairs of permutations. If the conjecture is true, then for any $n$ where $PM_n \neq DS_n$, there will exist an eigenvalue in $DS_n$ outside of $PM_n$ that comes from a pair of permutations, so that this approach will find a counterexample if the precision is taken fine enough. This approach dramatically reduces the problem size, and makes tractable the computation of $DS_n$ to reasonable precision for higher $n$. To control precision, we choose a \textit{mesh size}, meaning the number of matrices along a convex combination of two matrices whose eigenvalues we compute. For instance, for a mesh size of $11$, and for a pair of permutations $P$ and $Q$, we compute eigenvalues of $P, \frac{9}{10}P + \frac{1}{10}Q, \ldots,\frac{1}{10}P + \frac{9}{10}Q,$ and $Q$.

There are $O(\binom{n!}{2}) = O((n!)^2)$ pairs of permutations to consider. This can be significantly reduced with the following two lemmas. For each cycle type of $S_n$, choose a representative permutation matrix $C$ of that cycle type that is a direct sum of cyclic permutation matrices (in descending order of length). We call such a $C$ the \textit{canonical cycle form} of any permutation matrix representing a permutation of the same cycle type.

\begin{lemma}
Every pair of permutations $P_1, P_2$ is uniformly similar to a pair of permutations $C, R$, in which $C$ is the canonical cycle form of $P_1$.
\end{lemma}
\begin{proof}
Since $C$ is of the same cycle type as $P_1$, we can choose a permutation matrix $Q$ that conjugates $P_1$ to $C$. Then $C = QP_1 Q^T$ and $R = QP_2 Q^T$ gives the desired pair of permutations.
\end{proof}

\begin{lemma}\label{PermutationSimilarity}
Two permutation matrices are similar if and only if they are similar by a permutation similarity.
\end{lemma}
\begin{proof}
All permutation matrices are diagonalizable, being representations of $S_n$. Thus, two permutation matrices are similar if and only if they have the same multiset of eigenvalues. The multiset of eigenvalues is determined by cycle-type, and two permutation matrices are similar by a permutation similarity if and only if they have the same cycle-type.
\end{proof}

Thus, to compute the eigenvalues that are achievable by convex combinations of any pairs of permutation matrices, we need only consider eigenvalues achievable by convex combinations of each cycle type's canonical form with other permutation matrices. There are $p(n) \cdot n!$ such pairs, in which $p(n)$, the number of partitions of size $n$, counts the number of cycle types of permutations of size $n$. This also provides a simple reduction for convex combinations of more than two permutation matrices---for three matrices, $O(p(n)\cdot (n!)^2 )$ triples need to be computed if we consider triples consisting of each cycle type's canonical form and two other permutations in $S_n$.

In fact, the number of pairs to be considered can be reduced further with more precise utilization of uniform permutation similarity. For any pairs of permutations that are equal up to uniform conjugation, we need only consider one representative pair since similarity preserves eigenvalues. Consider the group $S_n \times S_n$, the direct product of the symmetric group with itself, and the action of diagonal conjugation on this group by $S_n$---meaning that the permutation $p$ acts on $(\sigma, \tau)$ by $(\sigma, \tau)^p = (p \sigma p^{-1}, p \tau p^{-1})$. Then we need only consider one representative pair from each orbit of this action on $S_n \times S_n$. Here we introduce some necessary concepts from group theory that we use to derive a computationally feasible method to obtain such representative pairs. See \cite{Bouc} for a reference for definitions and results.

A $G$-set is a set with a group action by a group $G$. Two $G$-sets $X$ and $Y$ are isomorphic, denoted $X \cong Y$, if there is a bijection $f: X \to Y$ such that $f(g\cdot x) = g\cdot f(x)$. Let $\bigsqcup$ denote the disjoint union operation on $G$-sets, and let $X \times Y$ denote the product of two $G$-sets, meaning the $G$-set that is the set given by the regular cartesian product of $X$ and $Y$ along with the action $g\cdot(x,y) = (g\cdot x, g\cdot y)$. For a subset $S \subseteq G$, let $\mc{O}(x)$ denote the orbit of an element $x$ under the action of conjugation on $S$ by $G$, $C_G(x)$ denote the centralizer in $G$ of an element $x$, and $\mc{C}$ denote a set of representatives of each conjugacy class of $G$.  Then we have the following isomorphisms of $G$-sets by basic properties of orbit-stabilizer relationships and distributivity:
\begin{align*}
G \times G \cong & \Big(\bigsqcup_{x \in \mc{C}} \mc{O}(x) \Big) \times \Big( \bigsqcup_{y \in \mc{C}} \mc{O}(y)\Big)\\
\cong & \Big(\bigsqcup_{x \in \mc{C}} G/C_G(x) \Big) \times \Big( \bigsqcup_{y \in \mc{C}} G / C_G(y)\Big)\\
\cong & \bigsqcup_{x \in \mc{C}} \; \bigsqcup_{y \in \mc{C}} \big(G/C_G(x) \big) \times \big( G/ C_G(y) \big)
\end{align*}

Now, for subgroups $H, K \leq G$, define the double coset $H g K = \{h g k : h \in H, k \in K\}$ for $g \in G$, the double coset space $H \backslash G / K = \bigcup_{g \in G} H g K$, and for any choice of representatives of each double coset in $H \backslash G / K$ denote the set of all representatives by $[H \backslash G / K]$. We will make use of Mackey's formula for $G$-sets, which can be found in \cite{Bouc}:

\[(G / H) \times (G/ K) \cong \bigsqcup_{g \in [H \backslash G / K]} G / (H \cap g K g^{-1}) \]

Applying Mackey's formula with $H = C_G(x)$ and $K = C_G(y)$ gives us that

\begin{align*}
G \times G \cong & \bigsqcup_{x \in \mc{C}} \; \bigsqcup_{y \in \mc{C}} \; \bigsqcup_{g \in [C_G(x) \backslash G / C_G(y)] } G/(C_G(x) \cap gC_G(y)g^{-1}) \\
\cong & \bigsqcup_{x \in \mc{C}} \; \bigsqcup_{y \in \mc{C}} \; \bigsqcup_{g \in [C_G(x) \backslash G / C_G(y)] } G/(C_G(x) \cap C_G(gyg^{-1}))
\end{align*}

Thus, with $G = S_n$, the representative pairs from the orbits of the action of diagonal conjugation on $S_n \times S_n$ are in one-to-one correspondence with the tuples $(x, y, g)$, where $x$ and $y$ are representatives of conjugacy classes of $S_n$ and $g$ is a representative of the double cosets of $C_G(x)$ and $C_G(y)$. Since the eigenvalues of convex combinations of $(\sigma, \tau)$ are equal to those of $(\tau, \sigma)$, we need only consider about half of these representative pairs. We define a relation on the set of pairs of permutations, where two pairs $(\sigma_1, \tau_1)$ and $(\sigma_2, \tau_2)$ are related if they are uniformly permutation similar to each other, or if they are uniformly permutation similar after a reversal of one, meaning $(\sigma_1, \tau_1)^p = (\tau_2, \sigma_2)$ for some $p \in S_n$. It can be seen that this relation is an equivalence relation---we call the distinct equivalence classes \textit{inequivalent pairs}.

There exist double-coset representative enumeration algorithms due to the study of computational group theory which we can take advantage of in our computation of the inequivalent pairs. We use routines from the GAP computer algebra system for double-coset representative enumeration and centralizer computations \cite{GAP}. Algorithm \ref{alg:1} is our algorithm for determining the inequivalent pairs.

\begin{algorithm} 
		\caption{Compute Inequivalent Pairs}
        \label{alg:1}
		\begin{algorithmic}[1]
				\For {$\sigma$ one representative of cycle type $i \in \{1, \ldots, p(n)\}$}
				    \For {$\tau$ one representative of cycle type $j \in \{i, \ldots, p(n)\}$}
				        \State Compute centralizers $C_G(\sigma)$ and $C_G(\tau)$
    				    \State Compute representatives of double cosets $[C_G(\sigma) \backslash G / C_G(\tau)]$
    				    \For {$g$ in $[C_G(\sigma) \backslash G / C_G(\tau)]$}
    				        \State Add $(\sigma, g\tau g^{-1})$ to \texttt{pairsList}
    				    \EndFor
				    \EndFor
				\EndFor
				\State \Return \texttt{pairsList}
		\end{algorithmic}
\end{algorithm}

In Table \ref{tab:numpairs} we list the number of inequivalent pairs in comparison to $p(n) \cdot n!$, the number of pairs considered under our initial method without this reduction. Let $a(n)$ be the number of equivalence classes of pairs in $S_n \times S_n$, where uniformly conjugate pairs are equivalent (this is also the number of orbits in $S_n \times S_n$ under the action of uniform conjugation by $S_n$), and $b(n)$ the number of such classes where any representative of the class is a pair $(\sigma, \tau)$ where $\sigma$ and $\tau$ have the same cycle type. Then the number of inequivalent pairs is exactly $(a(n) + b(n))/2$. $a(n)$ is known to grow at $O(n!)$ \cite{OEIS}, so the number of inequivalent pairs also grows at $O(n!)$. Thus the difference is substantial between the number of pairs considered with the naive method and the number of inequivalent pairs. Since Algorithm \ref{alg:1} computes the inequivalent pairs quickly for $n \leq 12$, we find it makes a major difference in our computation. Compared to computing the pairs of canonical cycle types with permutations, computing only eigenvalues for inequivalent pairs saves time and memory, so this reduction is useful for both searching for counterexamples and plotting in fine mesh sizes.

\begin{table}[ht]
\centering
\begin{tabular}{|c|c|c|}
        \hline
		$n$ & inequivalent pairs & $p(n) \cdot n!$\\
		\hline
		$2$ & $3$ & $4$\\
		$3$ & $8$ & $18$\\
		$4$ & $28$ & $120$\\
		$5$ & $98$ & $840$\\
		$6$ & $518$ & $7,\!920$\\
		$7$ & $3{,}096$ & $75{,}600$\\
		$8$ & $23{,}415$ & $887{,}040$\\
		$9$ & $201{,}795$ & $10{,}886{,}400$\\
		$10$ & $1{,}973{,}189$ & $152{,}409{,}600$\\
		$11$ & $21{,}347{,}935$ & $2{,}235{,}340{,}800$\\
		$12$ & $253{,}282{,}652$ & $36{,}883{,}123{,}200$\\
		$13$ & $3{,}263{,}902{,}430$ & $628{,}929{,}100{,}800$ \\
		\hline
\end{tabular}
\caption{Number of inequivalent pairs compared to $p(n) \cdot n!$, where $p$ is the partition function.}
\label{tab:numpairs}
\end{table}

Now, we discuss other details of our computations of $DS_n$. Since doubly stochastic matrices have real elements, their complex eigenvalues come in conjugate pairs, so we only consider those in the upper half plane. Moreover, since $\Pi_2 = [-1,1]$ is always included in $DS_n$ for $n \geq 2$, we need not consider any real eigenvalues since all of them are accounted for. Similarly, using the deflation as in Lemma \ref{StandardRep}, only $n-1$ by  $n-1$ matrices need to be handled to compute $DS_n$ as the shared eigenvalue of $1$ can be deflated away by the matrix in (\ref{deflation}). To determine whether an eigenvalue $\lambda$ is within the region, we first use Lemma \ref{inradiuslemma} to see that any eigenvalue $\lambda$ with magnitude $|\lambda| \leq \cos(\frac{\pi}{n})$ cannot be an exception to Perfect-Mirsky. Otherwise, for each $k \leq n$ we check which vertices of $\Pi_k$ are the ones that $\mrm{Re}(\lambda)$ lies between, by checking which of the intervals $\big[\cos(\frac{2\pi j}{k}), \cos(\frac{2\pi (j+1)}{k}) \big)$ is the one that contains $\mrm{Re}(\lambda)$. Then we check whether $\mrm{Im}(\lambda)$ satisfies the linear inequality determined by $\mrm{Re}(\lambda)$ and $j$ that defines this side of the polygon $\Pi_k$.

\section{Examination of \texorpdfstring{$DS_5$}{}}

In \cite{MR}, the following $5$-by-$5$ doubly stochastic matrix was noted to have eigenvalues outside of $PM_5$ for $t \in [.49, .51]$:

\[\begin{bmatrix}
0 & 0 & 0 & 1 & 0\\
0 & 0 & t & 0 & 1-t\\
0 & t & 1-t & 0 & 0\\
0 & 1-t & 0 & 0 & t\\
1 & 0 & 0 & 0 & 0\\
\end{bmatrix}\]

We note that that these matrices are given by convex combinations of pairs of matrices, corresponding to
\[t \begin{bmatrix}
0 & 0 & 0 & 1 & 0\\
0 & 0 & 1 & 0 & 0\\
0 & 1 & 0 & 0 & 0\\
0 & 0 & 0 & 0 & 1\\
1 & 0 & 0 & 0 & 0\\
\end{bmatrix} + (1-t) \begin{bmatrix}
0 & 0 & 0 & 1 & 0\\
0 & 0 & 0 & 0 & 1\\
0 & 0 & 1 & 0 & 0\\
0 & 1 & 0 & 0 & 0\\
1 & 0 & 0 & 0 & 0\\
\end{bmatrix}\]

\noindent in which the two permutations matrices correspond to $(145)(23)$ and $(1425)$ respectively. Using high-precision arithmetic, we have computed that the range of convex coefficients $t$ for which the exceptional eigenpath of this matrix has eigenvalues outside $PM_5$ is about [.4705275, .5490013]. The exceptional curve in the upper half plane connects a third root of unity at $e^{2\pi i/3}$ with a fourth root of unity at $i$. It leaves $PM_5$ near the intersection of $\Pi_3$ with $\Pi_4$ and barely stays within $PM_5$ near the intersection of $\Pi_4$ and $\Pi_5$. See Figures \ref{fig:zoomexception} and \ref{fig:zoomtriples} for images of this curve. Moreover, we note that only one other eigenpath is close to the boundary of $PM_5$. This path is given by a pairing between $(1425)$ and $(12345)$.

The pair of cycle types for the two permutations that generate the exceptional curve are of note. This exception is present in $DS_5$, and does not include a 5-cycle. Instead, the element of the pair that is present in $S_5$ but not in $S_4$ is of type 2,3. The following is an important observation that demonstrates why this curve is truly exceptional.

\begin{observation}
There is only one inequivalent pair, up to uniform permutation similarity, that generates an eigenvalue outside of $PM_5$.
\end{observation}

Take the original exceptional pair $(145)(23)$ and $(1425)$. We determine the inequivalent pair classes of this pairing of cycle types 2,3 with 1,4. Without loss of generality, we can fix $(1425)$, since if we want to determine the class of a pair with another 4-cycle, we can first take a uniform conjugacy to shift it to a pair with $(1425)$. Thus, any further conjugation must be by an elements of $C_{S_n}((1425)) = \langle (1425) \rangle$, the centralizer of $(1425)$. We can view the inequivalent pair classes of this pairing of cycle types as the orbits of the action of $C_{S_n}((1425))$ on the set $\mc{B}$ of elements of cycle type 2,3 by conjugation. The stabilizer of the action for any $\tau \in \mc{B}$ is trivial due to the cycle structures, so there are precisely $|\mc{B}| / |C_{S_n}((1425))| = 20/4 = 5$ orbits. Computing the eigenpaths of pairings $(r, (1425))$, where $r$ is a representative of each orbit, we find that only the inequivalent pair class of the original counterexample gives an eigenpath that leaves $PM_5$; in fact, all of the eigenpaths from other inequivalent pair classes stay well inside $PM_5$. Lastly, since the exceptional curve goes from a third root of unity to a fourth root of unity, we check pairings of $(1425)$ with 3-cycles (there are again $5$ inequivalent pair classes) and again find that no path even comes close to leaving $PM_5$.

In Table \ref{tab:pairs}, we list the inequivalent pair classes of pairings of permutations of cycles types 2,3 with 1,4. Understanding the differences between the inequivalent pair classes of pairs of these cycle types may be essential to understanding why the exceptional pair of permutations leaves $PM_5$.

\begin{table}[ht]
    \centering
\begin{tabular}{|c|c|}
\hline
     Class & 2,3 type Permutations of Class  \\
     \hline
     1 & $(34)(125),$ $(35)(142),$ $(23)(145),$ $(13)(254)$\\
     \hline
     2 & $(12)(345),$ $(12)(354),$ $(45)(132),$ $(45)(123)$\\
     \hline
     3 & $(35)(124),$ $(34)(152),$ $(13)(245),$ $(23)(154)$\\
     \hline 
     4 & $(24)(135),$ $(15)(234),$ $(25)(143),$ $(14)(253)$\\
     \hline
     5 & $(25)(134),$ $(14)(235),$ $(15)(243),$ $(24)(153)$\\
     \hline
\end{tabular}
\caption{Inequivalent pair classes and all representatives of each class with respect to pairing a 2,3 permutation with $(1425)$. To determine which inequivalent pair class a general pair $(\sigma, \tau)$ corresponds to, where $\sigma$ has type 2,3 and $\tau$ has type 1,4, uniformly conjugate $(\sigma, \tau)^p$ by a $p \in S_n$ which conjugates $\tau$ to $(1425)$, then check which row of the table $p\sigma p^{-1}$ is contained in. Note that the exceptional pair is in class $1$.}
\label{tab:pairs}
\end{table}

\begin{figure}
    \centering
    \includegraphics[scale=1]{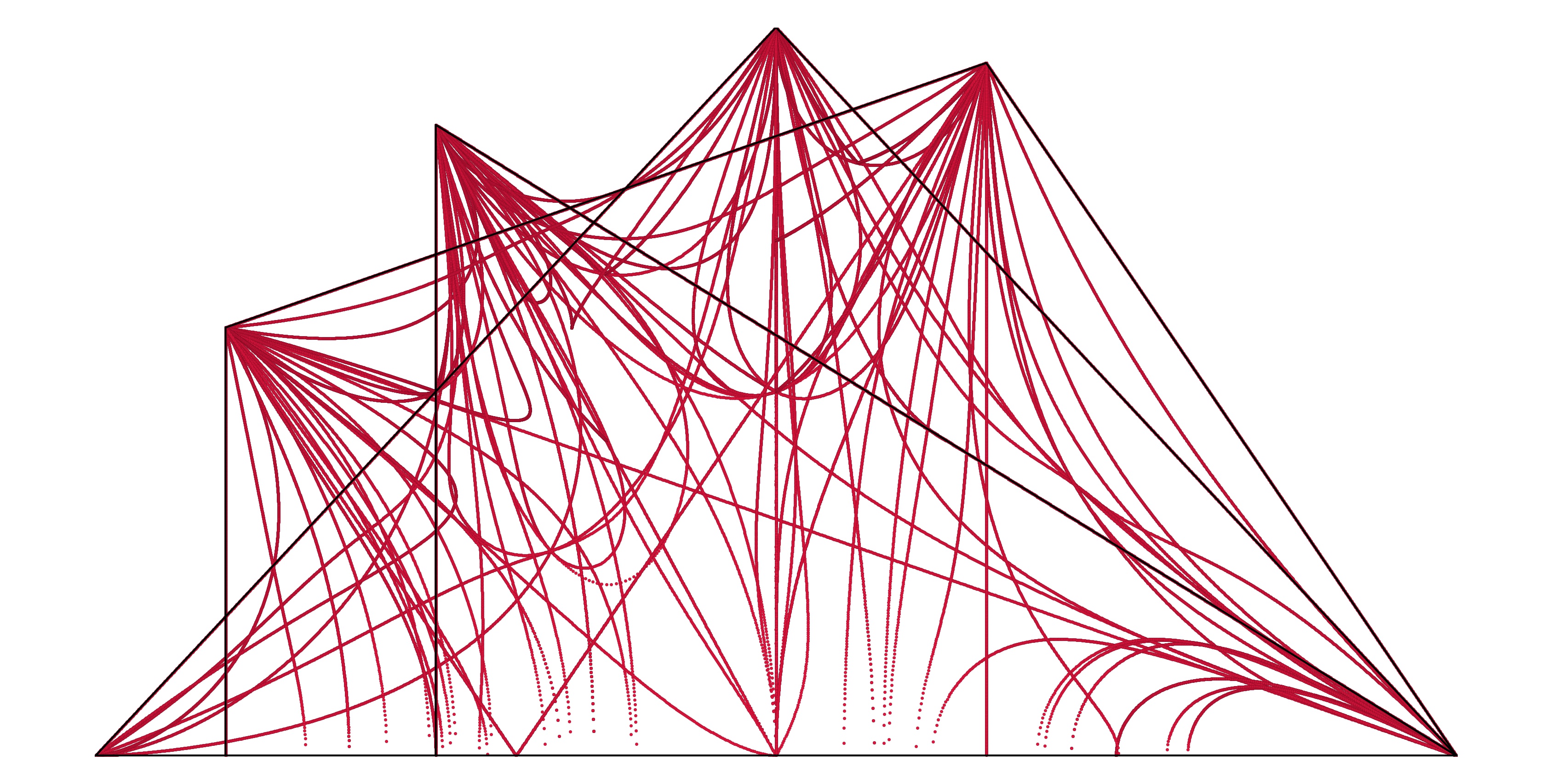}
    \caption{Inequivalent pairs for $DS_5$. The boundaries of $\Pi_k$ for $k \leq 5$ are outlined in black. Eigenvalues of inequivalent pairs are in red. Only the upper half plane is shown due to the symmetry of $DS_n$ across the real line.}
\end{figure}

\begin{figure}
    \centering
    \includegraphics[scale=1]{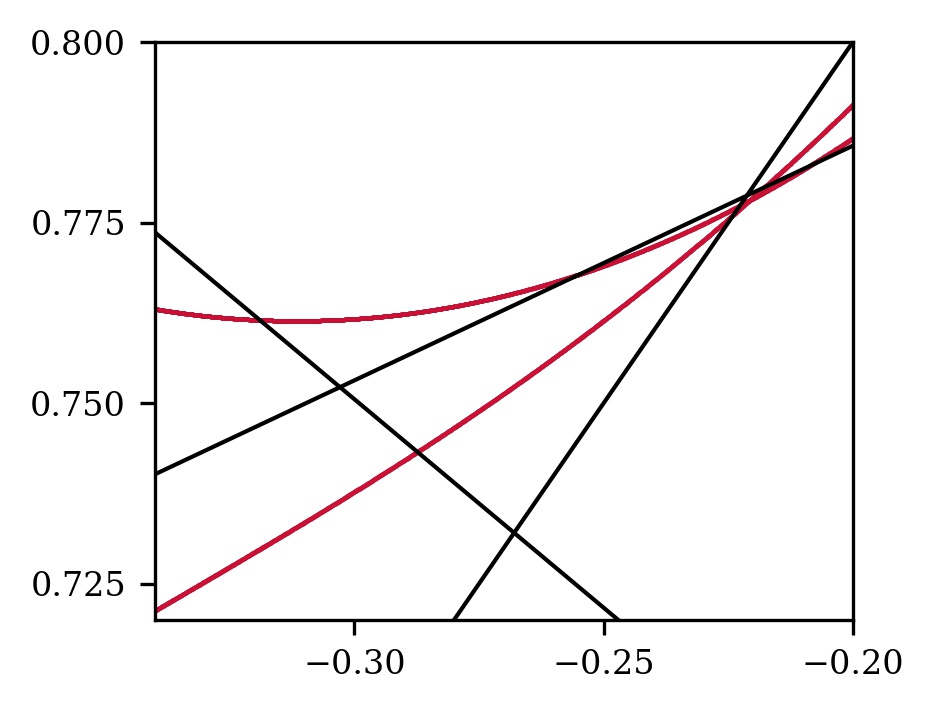}
	\caption{Close-up of exceptional curve in $DS_5$, along with other pair $(1234)$ that comes close to leaving $PM_5$, both in red. $PM_5$ is in black as per usual. The exceptional curve leaves $PM_5$ above the intersection of $\Pi_3$ and $\Pi_5$.}
    \label{fig:zoomexception}
\end{figure}

\begin{figure}
    \centering
    \includegraphics[scale=1]{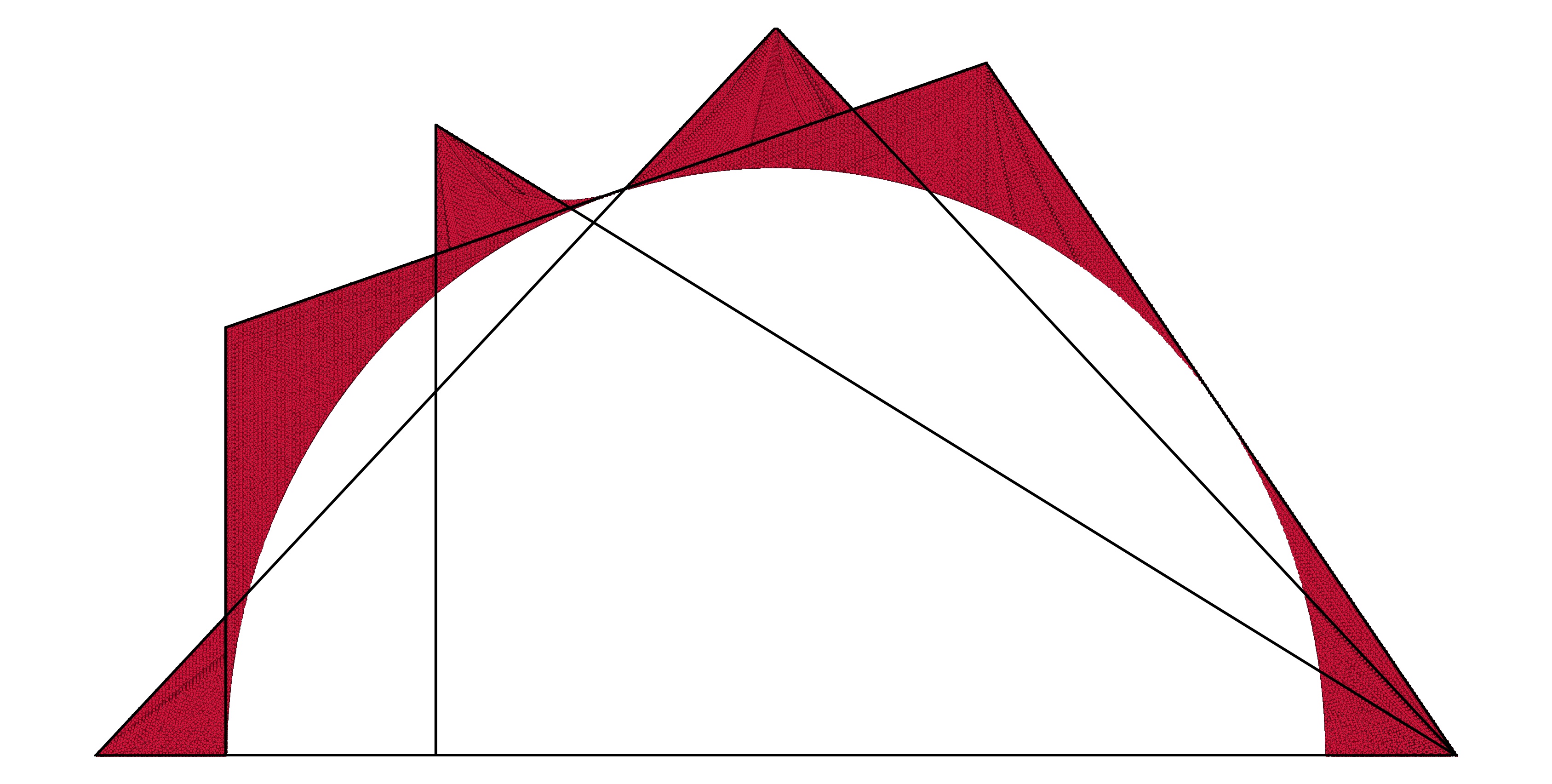}
    \caption{Eigenvalues of triples of permutations for $DS_5$. The inscribed circle is omitted.}
    \label{fig:triples}
\end{figure}

\begin{figure}
    \centering
    \includegraphics[scale=1]{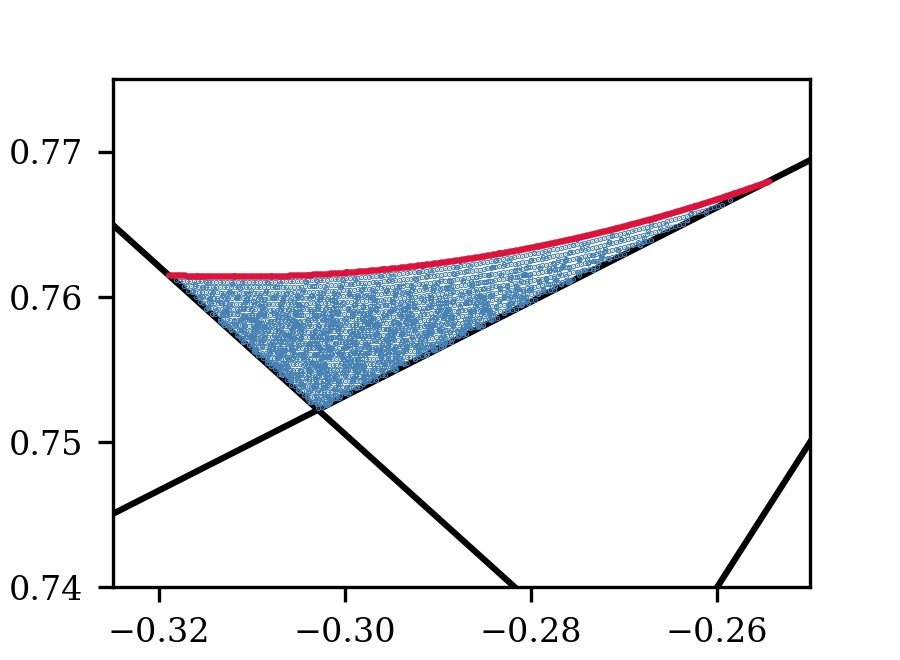}
	\caption{Close-up of subset of $DS_5$ generated by triples. Eigenvalues of triples are in blue while the exceptional curve is in red. Only those points outside of $PM_5$ are plotted. No eigenvalues of triples extend beyond the exceptional curve that is generated by pairs.}
    \label{fig:zoomtriples}
\end{figure}

Furthermore, we have computed convex combinations of triples of permutations in $\Omega_5$ to a mesh size of $200$ and for quadruples at a mesh size of $43$. For a mesh size of $m$ in the case of k-tuples, we mean that for every k-tuple of permutations considered, we take $m$ weights of size $1/m$ and compute eigenvalues for every possible distribution of these weights among the $k$ permutations. Thus, per triple of permutations, we compute eigenvalues for $O(m^2)$ matrices, as opposed to the $O(m)$ matrices computed along each pair.  As shown in Figures \ref{fig:triples} and \ref{fig:zoomtriples}, no eigenvalues of triples leave the region interior to the eigenvalues of pairs. Likewise, no eigenvalues of quadruples leave this region. This supports the Boundary Conjecture.

\section{Computational Results}

Our experiments have found no counterexamples to the Perfect-Mirsky conjecture for any $n \neq 5$. More specifically, there are no counterexamples along convex combinations of pairs of permutations for mesh size choices of $m = 10000$ for $n=6,7,8,9$, $m = 1000$ for $n=10$, and $m=200$ for $n=11$. Recall that the exceptional curve for $n=5$ lies outside of $PM_5$ for convex coefficients in an interval of length greater than $.07$, and counterexamples can be found with coarse mesh sizes---any mesh size larger than $14$ suffices to find one but we also note that any odd mesh size contains $1/2$ as a choice of convex coefficient and thus by chance happens to find a counterexample. Hence, the mesh sizes that we used for these computations for pairs with higher $n$ seem well-beyond sufficiently fine. However, we also note that it is reasonable to expect finer mesh sizes to be necessary for finding counterexamples for larger $n$, as $PM_n$ takes up more space in the unit disc and the absolute distance by which any exceptional curve leaves the region might be smaller.

\begin{figure}
    \centering
    \includegraphics[scale=1]{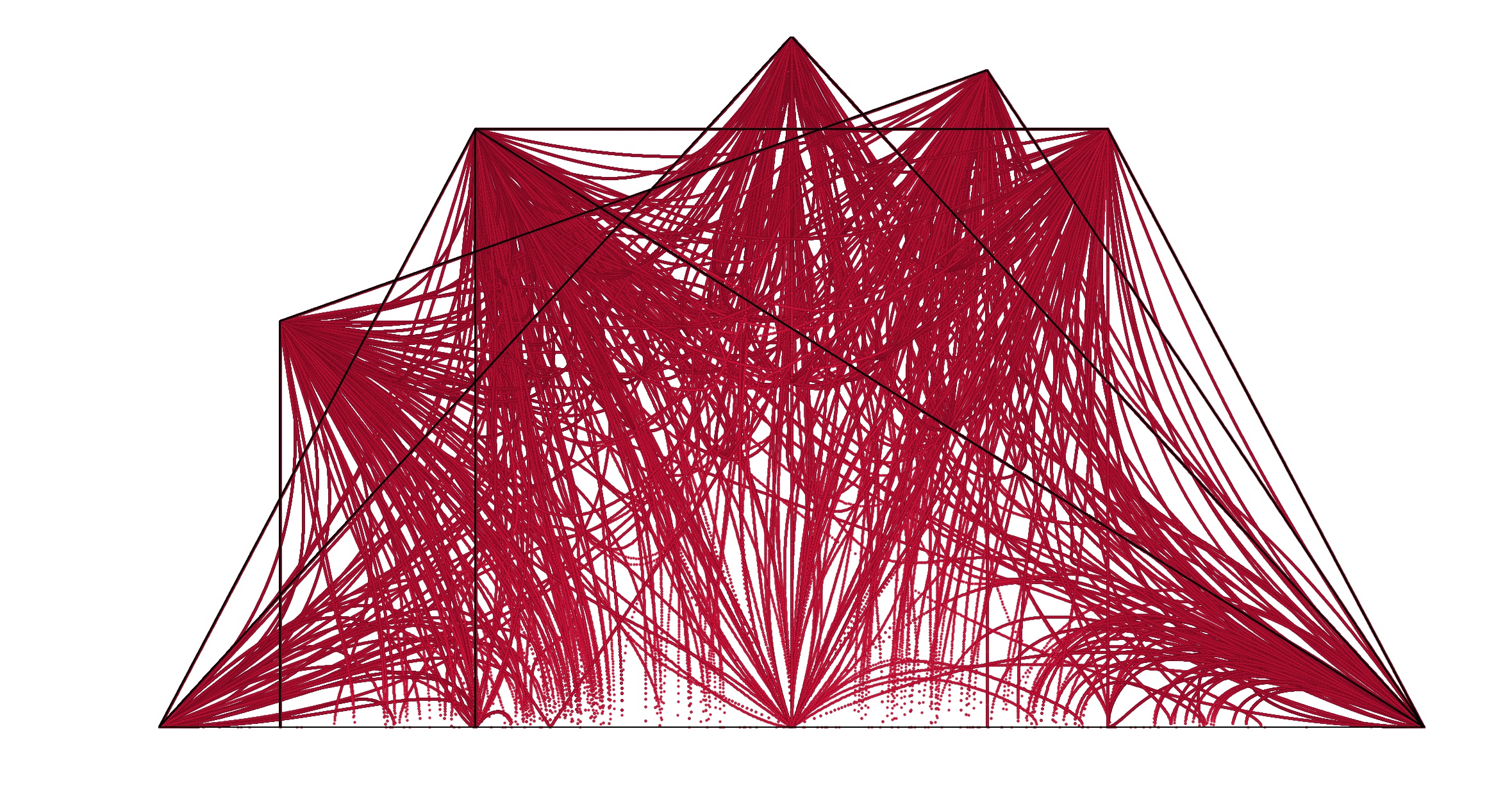}
    \caption{Inequivalent pairs for $DS_6$. $\Pi_k$ for $k \leq 6$ outlined in black.}
    \label{fig:DS_6}
\end{figure}

\begin{figure}
    \centering
    \includegraphics{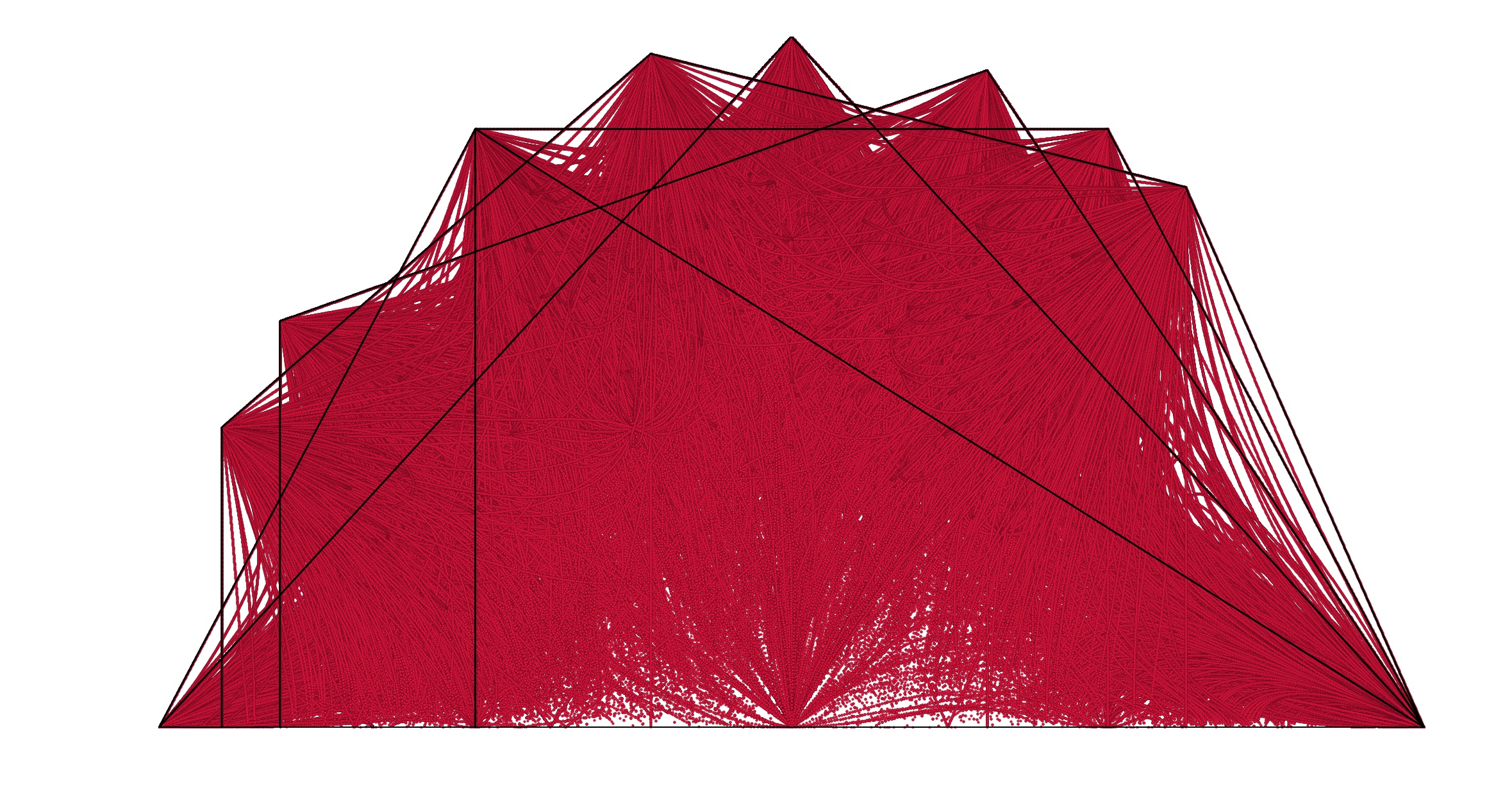}
    \caption{Inequivalent pairs for $DS_7$. $\Pi_k$ for $k \leq 7$ outlined in black.}
    \label{fig:DS_7}
\end{figure}

In Figures \ref{fig:DS_6} and \ref{fig:DS_7} we plot the (eigenvalues of) inequivalent pairs for $DS_6$ and $DS_7$. Every curve is quite far from leaving the boundary of $PM_n$, unlike the $n=5$ case, in which besides the exceptional curve, there was another curve quite close to the boundary of $PM_5$. Likewise, in plotting pairs for $DS_8$ and $DS_9$, there are no curves that are close to leaving $PM_n$. It is difficult to obtain information from plots of $DS_{10}$ since the margin between $PM_{10}$ and the unit disc is lower and also since the high number of points makes plotting at a high mesh size difficult. We computed some triples for $n \leq 11$ at various mesh sizes, and, as expected, due to the distance that pairs were from the the boundary of $PM_n$ for $n \leq 11$, found that there were no counterexamples to Perfect-Mirsky for these $n$.

\section{Spectra of Convex Hulls of Matrix Groups}

In \cite{HS}, the authors consider the so-called \textit{hull spectra} of matrix groups, defined for a matrix group $G$ as 
\[HS(G) := \{\lambda : \lambda \in \sigma(A), A \in Co(G)\}\]
Note that $DS_n = HS(G)$ when $G$ is the group consisting of the $n$-by-$n$ permutation matrices. They determined the hull spectra for abelian matrix groups, all representations of the dihedral group, and all representations of the quaternion group. In all of the cases in which the groups under consideration were finite, the boundary of the hull spectrum was achieved by eigenvalues from convex combinations of pairs of group elements. For infinite abelian groups, there is a possibility that the hull spectrum contains a set that is the interior of the unit disc together with a dense subset of the unit circle. In this case, pairs achieve all values of the boundary of the hull spectrum that are actually part of the hull spectrum i.e. the set $\partial HS(G) \cap HS(G)$. In any other case with infinite abelian groups, pairs determine the boundary.

However, it should be noted that we have found exactly two groups in which pairs do not determine the boundary of the hull spectrum. These two groups are $A_4$ and $A_5$ in the standard representation, and $A_n$ in the standard representation for higher $n$ may also have this property. The exact forms of $A_1, A_2$ and $A_3$ are trivial to determine (they are $\{1\}, \{1\},$ and $\Pi_3$ respectively), and obviously have pairs determining the boundary. However, for $A_4$ and $A_5$, some eigenvalues obtained by convex combinations of triples of even permutations are not within the region interior to the eigenvalues achieved by convex combinations of pairs. See Figures \ref{A4} and \ref{A5} for visualization of these phenomena.

While the fact that the several classes of groups considered in \cite{HS} with known hull spectra have boundary determined by pairs does support the Boundary Conjecture, it is somewhat worrying that the standard representations of $A_4$ and $A_5$, which are of course closely related to the standard representations of $S_4$ and $S_5$, provide exceptions to the analogous conjecture for general matrix groups. However, we do believe that our arguments in Section 3 are still strong, and that along with the computational evidence that we have produced, suggests to us that the Boundary Conjecture is likely to be true.

\begin{figure}
    \centering
    \includegraphics[scale=1]{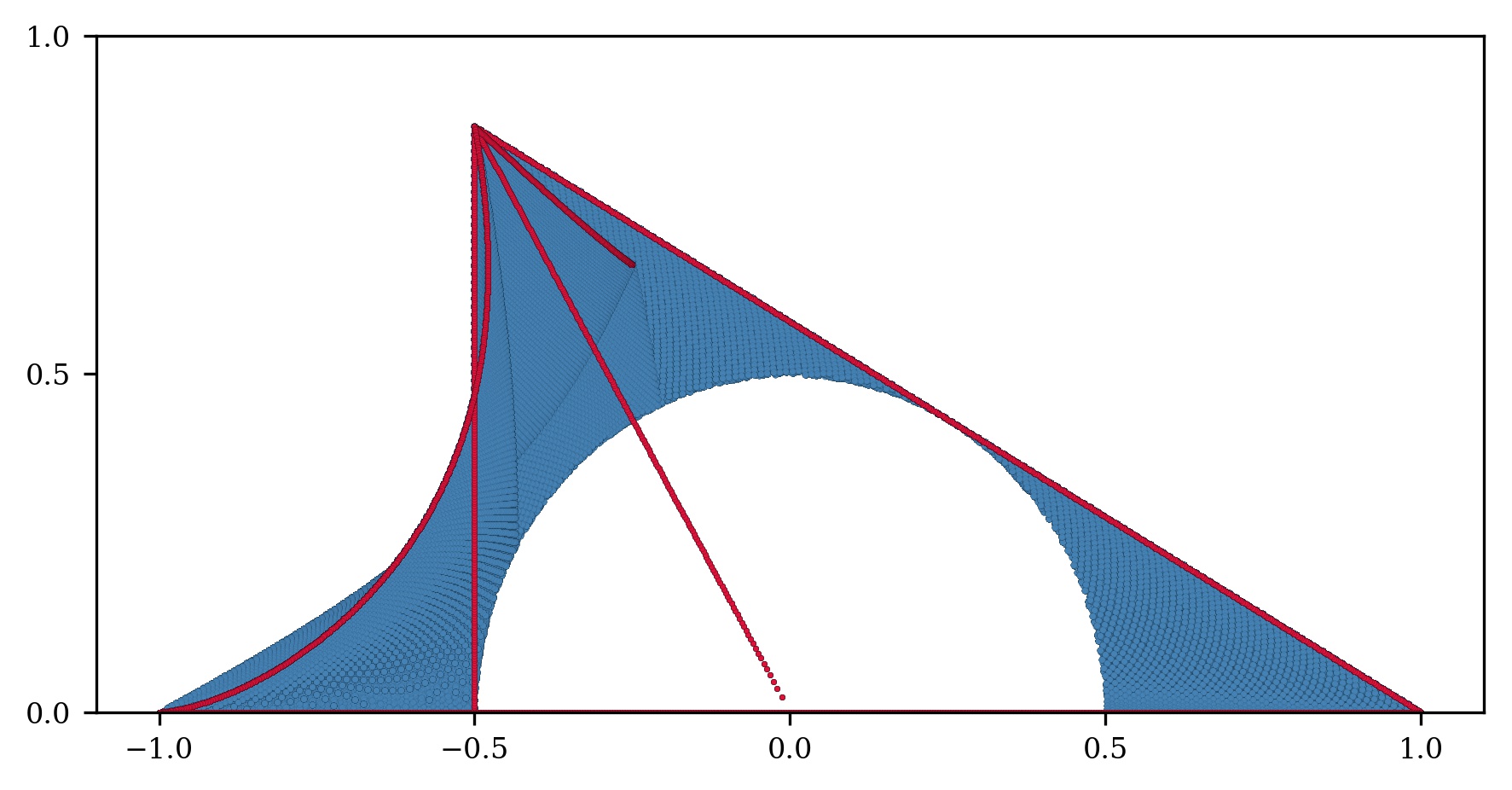}
    \caption{Eigenvalues of convex combinations of pairs and triples of $4$-by-$4$ even permutation matrices. Pairs are in red, and triples are in blue.}
    \label{A4}
\end{figure}

\begin{figure}
    \centering
    \includegraphics{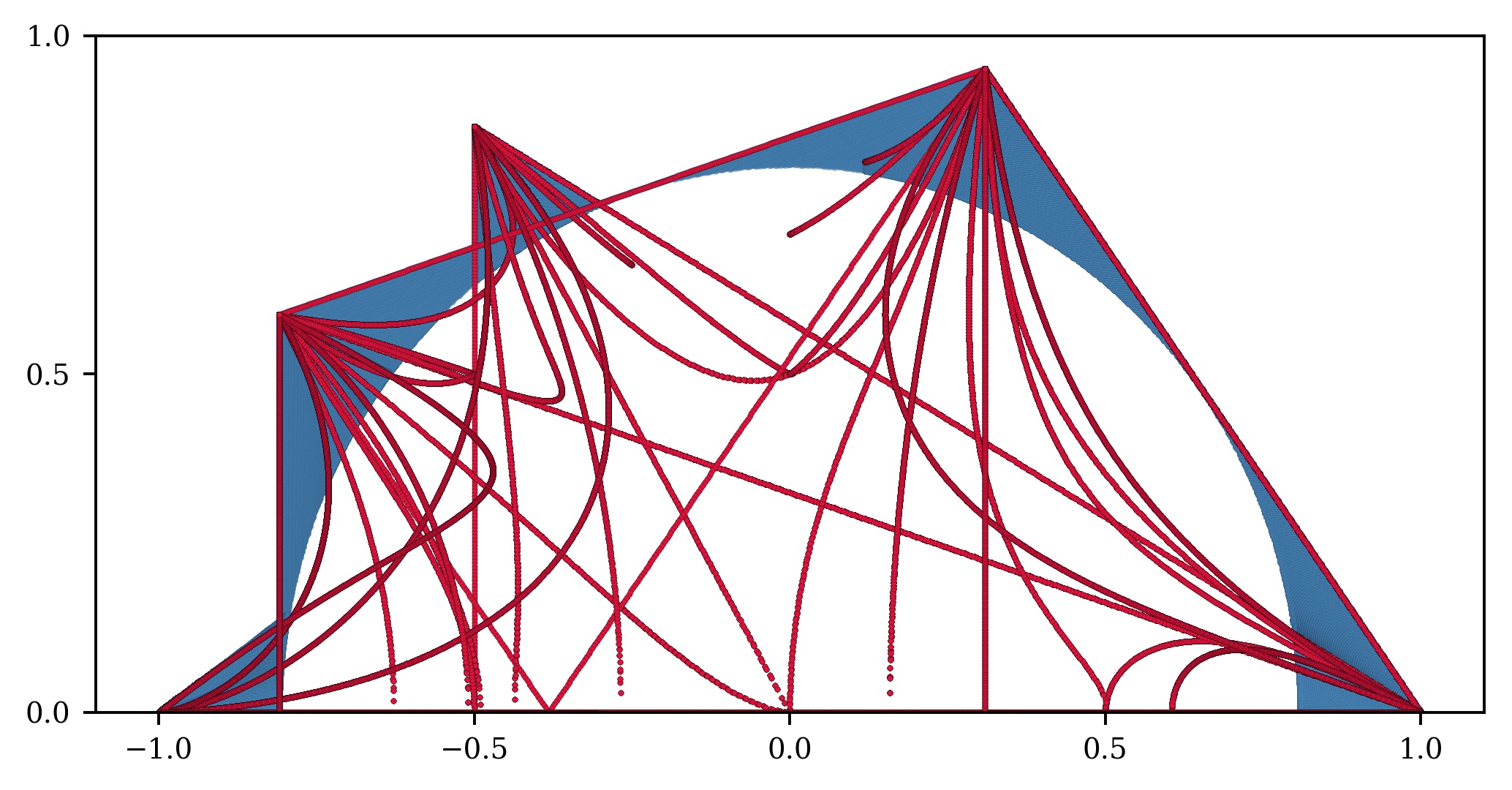}
    \caption{Eigenvalues of convex combinations of pairs and triples of $5$-by-$5$ even permutation matrices. Pairs are in red, and triples are in blue.}
    \label{A5}
\end{figure}

\begin{figure}
    \centering
    \includegraphics{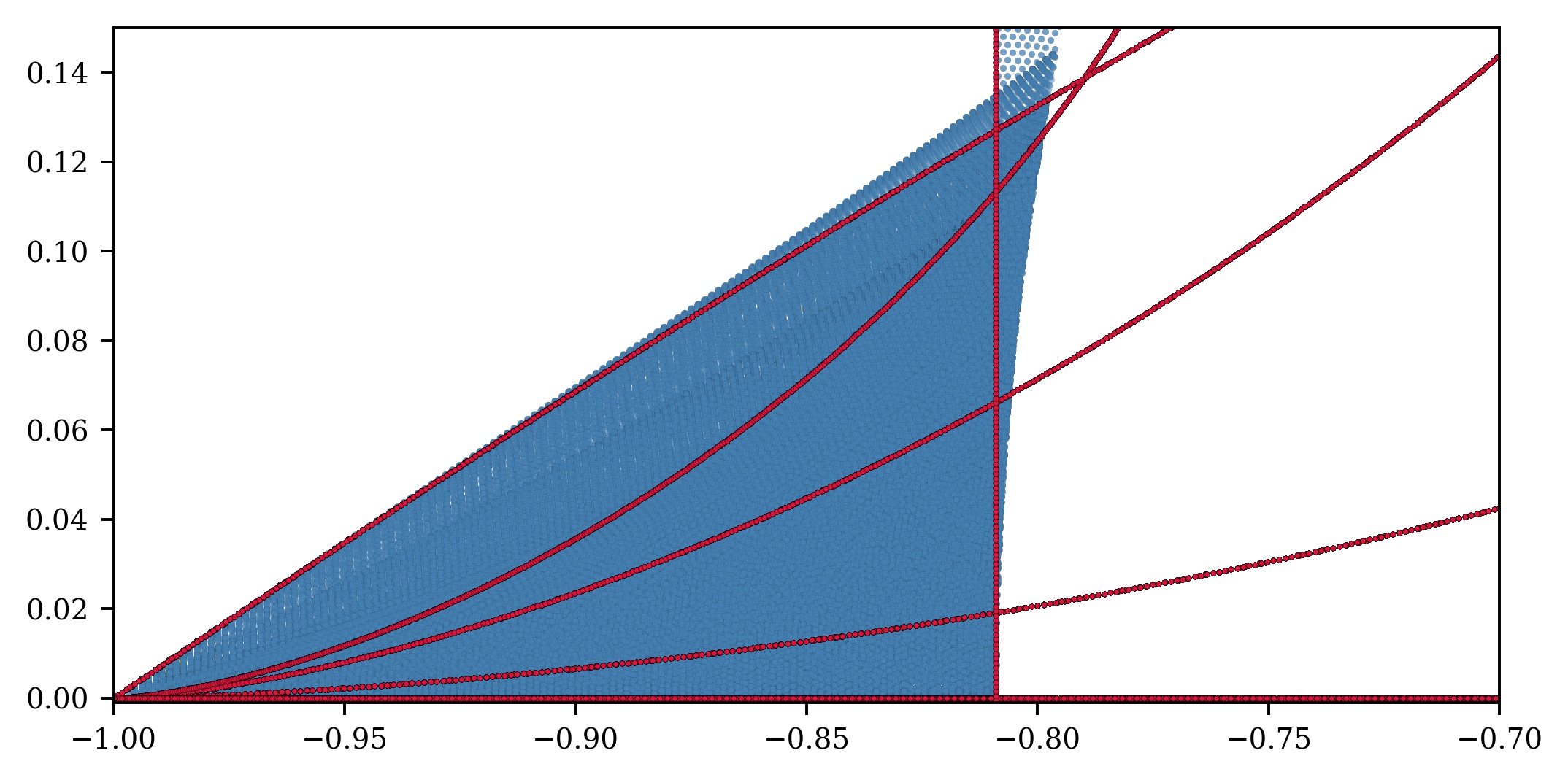}
    \caption{Zoomed in version of Figure \ref{A5}.}
    \label{A5zoom}
\end{figure}

\section{Further Discussion and Conclusion}

In \cite{DS4}, it was conjectured that $DS_n$ is the union $PM_n \cup K_{n-1}$. While this does hold for $n \leq 4$, since $K_{n-1}$ for such $n$ coincides with $DS_{n-1}$, our computations suggest that this is false in general. Our computations for triples in $DS_n$ for $n \geq 6$ show that there are no eigenvalues that leave $PM_{n}$, let alone come close to the arcs that form the boundary of the corresponding Karpelevich regions. For $n=5$, the exceptional curve stays well within the Karpelevich region $K_4$, and we do not find eigenvalues that leave $PM_{5}$ in other regions where $K_4$ extends beyond $PM_5$. See Figures \ref{fig:K4} and \ref{fig:K5} for visualizations of $PM_n \cup K_{n-1}$ for $n=5,6$. In fact, we believe in a stronger inclusion for certain $n$ that we have sufficiently computed. Our computational evidence, as outlined above, is consistent with the Perfect-Mirsky conjecture for certain $n$, so we make the following conjecture:

\begin{Conj}
$DS_n = PM_n$ for $n=6,7,8,9,10$ and $11$.
\end{Conj}

If the Boundary Conjecture is true, there would be little doubt in our conjecture due to the fine mesh sizes to which we have computed pairs. If the Boundary Conjecture is not true for some $n$, then we still believe that this conjecture is likely to be true for the values of $n$ we have discussed, since we have also computed triples for these $n$ and found no counterexamples.

\begin{figure}
    \centering
    \includegraphics[scale=1]{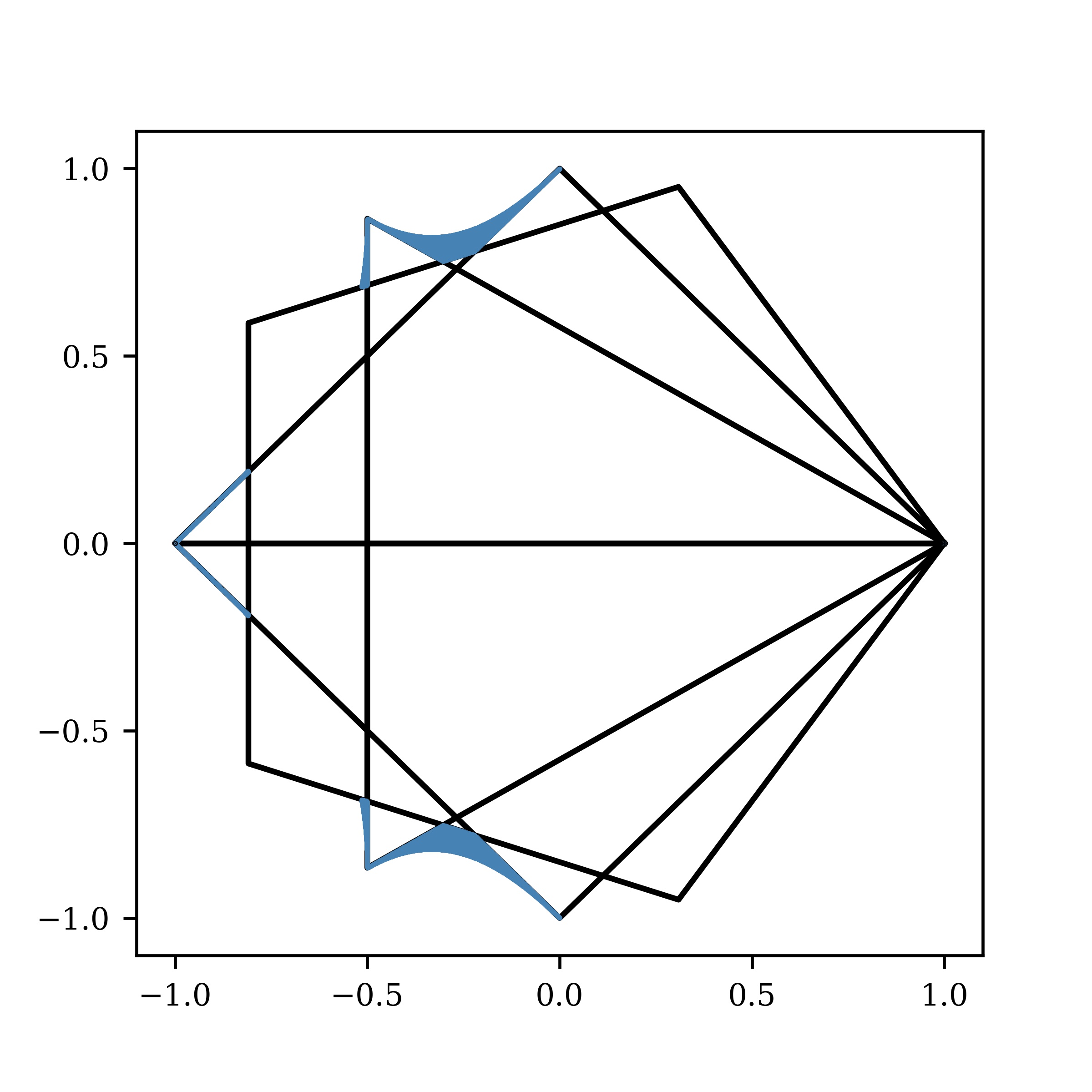}
    \caption{$PM_5$ in outlined in black and $K_4$ filled in blue. Only the subset of $K_4$ that is outside of $PM_5$ is displayed. See \cite{Swift} for equations determining $K_n$ for small $n$.}
    \label{fig:K4}
\end{figure}
\begin{figure}
    \centering
    \includegraphics[scale=1]{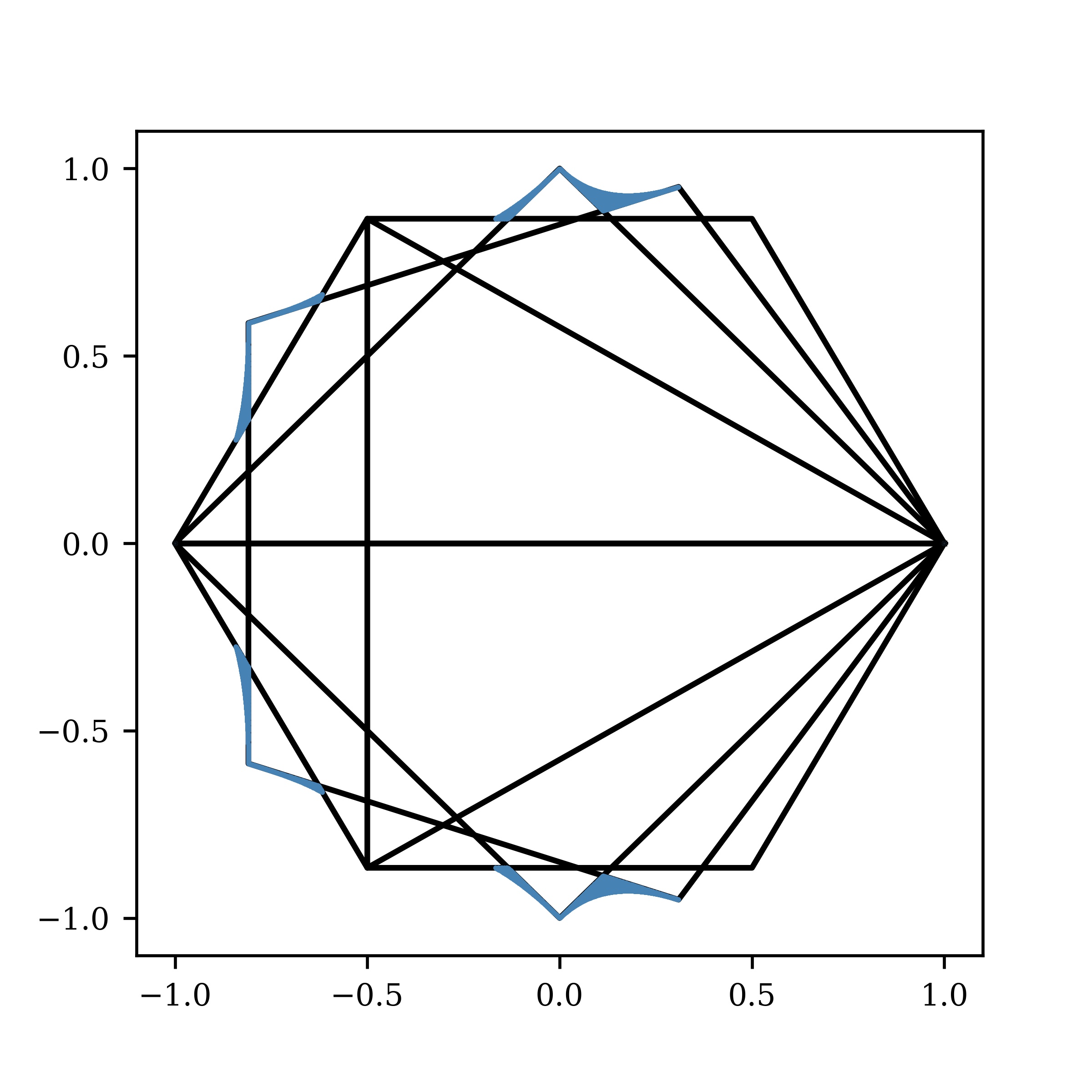}
    \caption{$PM_6$ in black and $K_5$ in blue. Again, only the subset of $K_5$ outside of $PM_6$ is displayed.}
    \label{fig:K5}
\end{figure}

We have spent much CPU time ($> 2$ years) on searching for exceptions to the Perfect-Mirsky conjecture for larger $n$, but have not found any. Our other methods for finding counterexamples did not result in any, but are worth noting. Since the original $n=5$ counterexample occurs as an eigenvalue of the average of two permutations, we checked whether eigenvalues of averages of pairs of permutations gave exceptions to Perfect-Mirsky. Also, we have tried computing eigenvalues of convex combinations of some triples of permutations for $n \leq 11$ and varying mesh sizes, but have not found any counterexamples in this way. The size of this problem of computing $DS_n$ to a given mesh size scales dramatically in $n$, and for higher-order tuples (beyond pairs) of permutations also scales quickly in the mesh size (recall that for triples, computing on a mesh size of $m$ means computing $O(m^2)$ matrices for each triple as opposed to $O(m)$ for each pair in the two matrix case). Thus, our computations are by no means absolutely comprehensive, but the computational results that we have obtained are consistent with the Perfect-Mirsky conjecture for these $n$ under consideration.

In the future, we hope that researchers with the right tools may find results that determine $DS_n$ for some $n > 4$. We do believe that Perfect-Mirsky holds for $n=6,7,8,9,10,11$, and have presented our computational results that are consistent with this assertion. In both of the works \cite{PM} and \cite{DS4} which made major advances in the study of $DS_n$, properties of doubly stochastic matrices were used extensively. For the most part, we work directly with the permutations that are the extreme points of the set of doubly stochastic matrices and are still able to deduce properties of $DS_n$. Algebraic properties of permutations could provide useful information if combined with the properties of doubly stochastic matrices in future study. In $DS_5$, the differences in behavior of eigenpaths of different inequivalent pairs are mysterious and likely worth studying.

\section{Acknowledgments}

We would like to thank the anonymous reviewer for his suggestions on the manuscript. We would also like to thank Jamie Barsotti for assistance with bringing our attention to the relevant theory on $G$-sets, and for assistance in deriving the algorithm to compute the inequivalent pairs. Moreover, we would like to thank John Wilkes, who carried out initial work on this problem in his honors thesis at the College of William and Mary advised by Charles R. Johnson \cite{Wilkes}.

This work was performed [in part] using computing facilities at the College of William and Mary, which were provided by contributions from the National Science Foundation, the Commonwealth of Virginia Equipment Trust Fund and the Office of Naval Research.

This work was supported by the National Science Foundation under Grant DMS \#1757603.

\bibliographystyle{siam}
\bibliography{main}{}

\appendix
\appendixpage
\section{Supplementary Figures}

\begin{figure}[ht]
    \centering
    \includegraphics{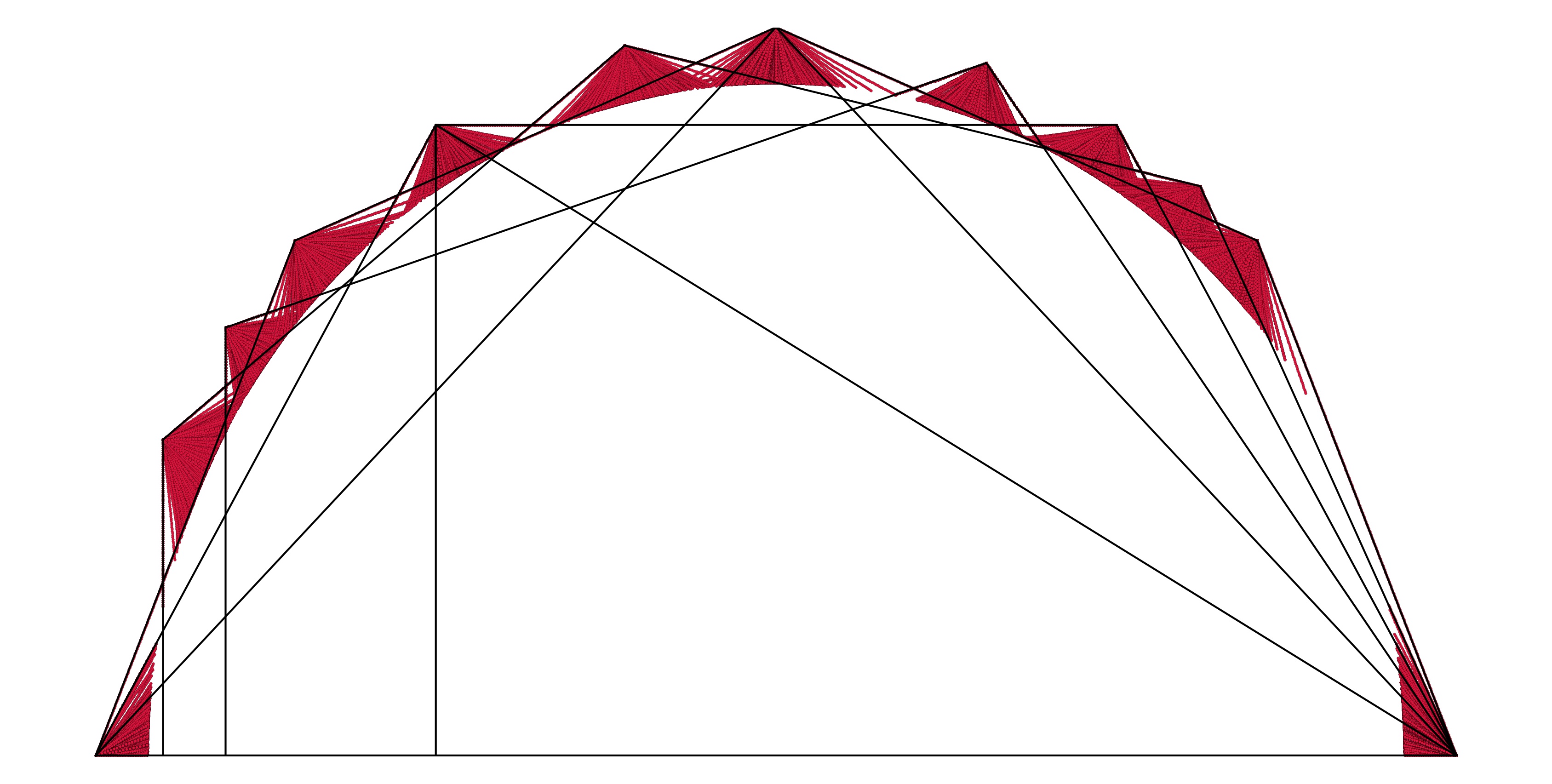}
    \caption{Inequivalent pairs for $DS_8$. Inscribed circle omitted.}
\end{figure}

\begin{figure}[ht]
    \centering
    \includegraphics{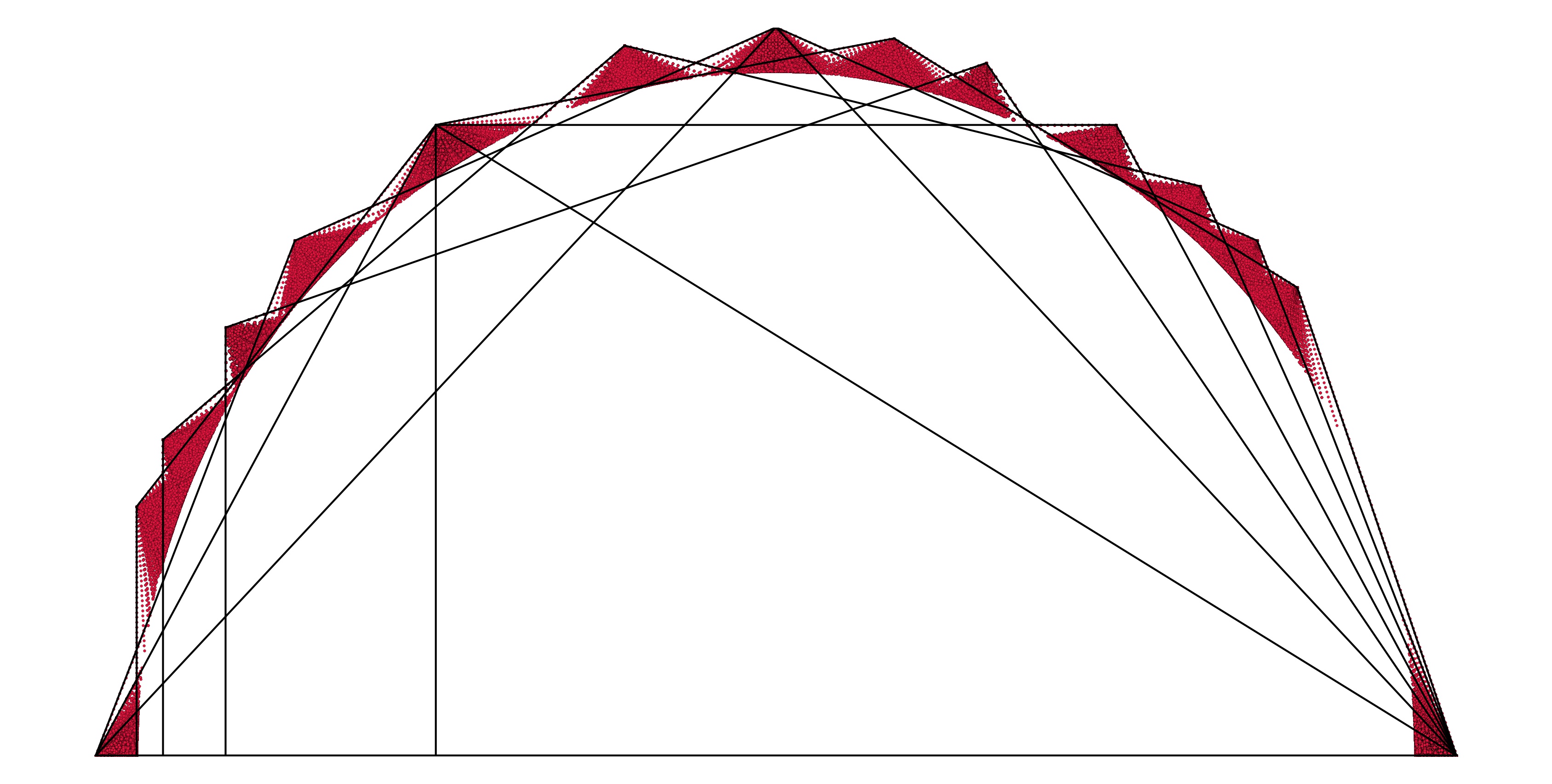}
    \caption{Inequivalent pairs for $DS_9$. Inscribed circle omitted.}
\end{figure}

\end{document}